\documentclass[12pt,reqno]{amsart}

\usepackage{amssymb}

\usepackage{hyperref}

\numberwithin{equation}{section}

\newtheorem{theorem}{Theorem}[section]
\newtheorem{proposition}[theorem]{Proposition}
\newtheorem{lemma}[theorem]{Lemma}
\newtheorem{corollary}[theorem]{Corollary}

\theoremstyle{definition}
\newtheorem{definition}[theorem]{Definition}
\newtheorem{remark}[theorem]{Remark}

\begin{document}

\baselineskip=15pt

\title[Criterion for Parabolic Lie algebroid connections]{A criterion for parabolic vector
bundles to admit a parabolic Lie algebroid connection}

\author[D. Alfaya]{David Alfaya}

\address{Department of Applied Mathematics and Institute for Research in Technology, ICAI 
School of Engineering, Comillas Pontifical University, C/Alberto Aguilera 25, 28015 Madrid, 
Spain}

\email{dalfaya@comillas.edu}

\author[A. Bansal]{Ashima Bansal}

\address{Department of Mathematics, Shiv Nadar University, NH91, Tehsil
Dadri, Greater Noida, Uttar Pradesh 201314, India}

\email{ashima.bansal@snu.edu.in}

\author[I. Biswas]{Indranil Biswas}

\address{Department of Mathematics, Shiv Nadar University, NH91, Tehsil Dadri,
Greater Noida, Uttar Pradesh 201314, India}

\email{indranil.biswas@snu.edu.in, indranil29@gmail.com}

\author[A. Singh]{Anoop Singh}

\address{Department of Mathematical Sciences, Indian Institute of Technology (BHU), Varanasi 
221005, India}

\email{anoopsingh.mat@iitbhu.ac.in}



\date{}

\begin{abstract}
Given a holomorphic Lie algebroid $(V,\, \phi)$ on a compact connected Riemann surface $X$,
we give a necessary and sufficient condition for a parabolic vector bundle on $X$, with parabolic
structure over a nonzero reduced effective divisor, to admit a parabolic Lie algebroid connection for
the Lie algebroid $(V,\, \phi)$. 
\end{abstract}

\maketitle

\section{Introduction}

Let $X$ be a compact connected Riemann surface. A holomorphic Lie algebroid on $X$ is a pair
of the form $(V,\, \phi)$, where $V$ is a holomorphic vector bundle over $X$ whose sheaf of
holomorphic sections is equipped with a structure of a complex Lie algebra and $\phi\, :\, V\,
\longrightarrow\, TX$ is a holomorphic homomorphism of vector bundles such that
$$[s,\, f\cdot t]\ =\ f\cdot [s,\, t]+\phi(s)(f)\cdot t$$ for all locally defined holomorphic sections
$s,\, t$ of $V$ and all locally defined holomorphic functions $f$ on $X$.

Fix a holomorphic Lie algebroid $(V,\, \phi)$ on $X$, and take any holomorphic vector bundle $E$ on $X$.
A holomorphic Lie algebroid connection on $E$ is a holomorphic differential operator
$$
D\,\,:\,\, E\,\,\longrightarrow\, \, E\otimes V^*
$$
satisfying the Leibniz identity which says that $D(fs) \,=\, fD(s) + s\otimes \phi^*(df)$
for all locally defined holomorphic sections $s$ of $E$ and all locally defined holomorphic
functions $f$ on $X$, where $\phi^*$ is the dual of the above homomorphism $\phi$ (See \cite{AO},
\cite{Al}, \cite{To1}, \cite{To2}, \cite{CM}, \cite{ELW}, \cite{LM}, \cite{PW} for Lie algebroids
and Lie algebroid connections).

Fix a nonzero reduced effective divisor $S$ on $X$. Let $E_*$ be a parabolic vector bundle on $X$ with parabolic 
structure over the divisor $S$. A quasi-parabolic Lie algebroid connection on $E_*$ for $(V,\, \phi)$ is a 
holomorphic Lie algebroid connection $D$ on the holomorphic vector bundle $E$ underlying $E_*$ for $(V,\, 
\phi)$ such that $D$ induces a holomorphic Lie algebroid connection (for $(V,\, \phi)$) on each of the 
subsheaves of $E$ given by the quasi-parabolic filtrations of $E_*$. A parabolic Lie algebroid connection on 
$E_*$ is a quasi-parabolic Lie algebroid connection $D$ on $E_*$ such that the second fundamental forms of $D$, 
for the subsheaves of $E$ given by the quasi-parabolic filtrations of $E_*$, is induced by multiplication by 
the parabolic weights.

In \cite{ABKS}, the following necessary condition for $E_*$ to admit a parabolic Lie algebroid connection
for $(V,\,\phi)$ was obtained: If $E_*$ admits a parabolic Lie algebroid connection
for the $(V,\,\phi)$, then the homomorphism $\phi$ vanishes on every parabolic point of $E_*$, namely the
points of the divisor $S$.

Assume that $\phi$ vanishes on every point of the parabolic divisor $S$. Let
$$
\phi_0\ :\ V \ \longrightarrow\ TX\otimes {\mathcal O}_X(-S)
$$
be the homomorphism given by $\phi$. The holomorphic Lie algebroid $(V,\, \phi)$ is called
logarithmically split if there is a holomorphic homomorphism of vector bundles
$$
\sigma\ :\ TX\otimes{\mathcal O}_X(-S)\ \longrightarrow\ V
$$
such that $\phi_0\circ\sigma\,=\, {\rm Id}_{TX\otimes{\mathcal O}_X(-S)}$. Notice that $\sigma$ is not required
to preserve the Lie bracket structures of $TX\otimes{\mathcal O}_X(-S)$ and $V$; it is just a holomorphic
vector bundle homomorphism.

The holomorphic Lie algebroid $(V,\, \phi)$ is called
logarithmically non-split if it is not logarithmically split.

We prove the following criterion for the existence of a parabolic Lie algebroid
connection (see Corollary \ref{cor2}):

\begin{theorem}\label{cori}
Let $(V,\, \phi)$ be a holomorphic Lie algebroid on $X$ satisfying the condition 
that $\phi$ vanishes over every parabolic point $x\, \in\, S$.
A parabolic vector bundle $E_*$ on $X$ admits a parabolic Lie algebroid connection
for $(V,\, \phi)$ if and only if at least one of the following two conditions holds:
\begin{enumerate}
\item The holomorphic Lie algebroid $(V,\, \phi)$ is logarithmically non-split.

\item The parabolic degree of every indecomposable component of $E_*$ is zero.
\end{enumerate}
\end{theorem}

In particular, we prove the following (see Corollary \ref{cor1}):

\begin{theorem}\label{thmi}
Let $(V,\, \phi)$ be a holomorphic Lie algebroid on $X$ satisfying the condition that $\phi$
vanishes over every parabolic point $x\, \in\, S$, and $(V,\,\phi)$ is logarithmically
non-split. Then any parabolic vector bundle $E_*$ on $X$ with parabolic divisor $S$
admits a parabolic Lie algebroid connection for $(V,\, \phi)$.
\end{theorem}

In \cite{ABKS}, Theorem \ref{cori} was proved under the assumption that the vector bundle
$V$ stable. Since a stable vector bundle is indecomposable, if $(V,\, \phi)$ is a
logarithmically split holomorphic Lie algebroid, and $V$ is stable, then $\phi$ is an
isomorphism of $V$ with $TX\otimes{\mathcal O}_X(-S)$.

Finally, we also provide some conditions for a parabolic vector bundle to admit a flat (same as integrable)
Lie algebroid connection compatible with the parabolic structure (see Corollary \ref{cor:integrable}).

\begin{theorem}\label{thmii}
Let $(V,\, \phi)$ be a holomorphic Lie algebroid on $X$ satisfying the condition that $\phi$
vanishes over every parabolic point $x\, \in\, S$, and let $E_*$ be a parabolic vector bundle. Then
the following two statements hold:
\begin{enumerate}
\item the parabolic vector bundle $E_*$ admits a flat quasi-parabolic Lie algebroid connection.
\item If the image of $\phi$ is either
\begin{enumerate}
\item strictly contained in $TX\otimes \mathcal{O}_X(-S)$, or
\item equal to $TX\otimes \mathcal{O}_X(-S)$, and the parabolic degree of every indecomposable component of
$E_*$ is zero,
\end{enumerate}
then $E_*$ admits a flat parabolic Lie algebroid connection.
\end{enumerate}
Moreover, if the Lie algebroid is logarithmically split, then (2.b) is actually a necessary and
sufficient condition for $E_*$ to admit a flat parabolic Lie algebroid connection.
\end{theorem}

\section{Lie algebroid connections and parabolic bundles}

\subsection{Holomorphic Lie algebroid connection}

Let $X$ be a compact connected Riemann surface. The holomorphic cotangent and tangent bundles
of $X$ will be denoted by $K_X$ and $TX$ respectively. The first
holomorphic jet bundle of a holomorphic vector bundle
$W$ on $X$ will be denoted by $J^1(W)$; so $J^1(W)$ is a holomorphic vector bundle on $X$
that fits in the following short exact sequence of holomorphic vector bundles on $X$:
$$
0\, \longrightarrow\, W\otimes K_X \, \longrightarrow\, J^1(W) \, \longrightarrow\, W
\, \longrightarrow\, 0.
$$

A $\mathbb{C}$--Lie algebra structure on a holomorphic vector
bundle $V$ on $X$ is a $\mathbb{C}$--bilinear pairing defined by a sheaf homomorphism
$$
[-,\, -] \,\,:\,\, V\otimes_{\mathbb C} V \,\, \longrightarrow\,\, V,
$$
which is given by an ${\mathcal O}_X$--linear holomorphic
homomorphism $$J^1(V)\otimes J^1(V)\, \longrightarrow\, V$$ of vector bundles, such
that $$[s,\, t]\,=\, -[t,\, s]\ \ \text{ and }\ \ [[s,\, t],\, u]+[[t,\, u],\, s]+[[u,\, s],\, t]\,=\,0$$
for all locally defined holomorphic sections $s,\, t,\, u$ of $V$.

The Lie bracket operation on $TX$ gives the structure of a $\mathbb C$--Lie algebra on it. More
generally, for any effective divisor $S$ on $X$, the holomorphic line bundle $TX\otimes {\mathcal
O}_X(-S)$ has the structure of a $\mathbb C$--Lie algebra given by the operation of Lie bracket
of vector fields.

A \textit{holomorphic Lie algebroid} on $X$ is a pair $(V,\, \phi)$, where
\begin{enumerate}
\item $V$ is a holomorphic vector bundle on $X$ equipped with the structure of a $\mathbb{C}$--Lie algebra,

\item $\phi \, : \, V \longrightarrow TX$ is a holomorphic vector bundle homomorphism, and

\item $[s,\, f\cdot t]\,=\, f\cdot [s,\, t]+\phi(s)(f)\cdot t$ for all locally defined holomorphic sections
$s,\, t$ of $V$ and all locally defined holomorphic functions $f$ on $X$.
\end{enumerate}
The above homomorphism $\phi$ is called the \textit{anchor map} of the Lie algebroid. As a consequence of (3),
this $\phi$ is a Lie algebroid map (see \cite[Remark 2.1]{ABKS}).

Let $(V,\, \phi)$ be a holomorphic Lie algebroid on $X$. We have the dual homomorphism
\begin{equation}\label{e2}
\phi^*\ :\ K_X\ \longrightarrow\ V^*
\end{equation}
of $\phi$. Let
\begin{equation}\label{e3}
q\,\, :\,\, V^* \,\, \longrightarrow\,\, V^*/\phi^*(K_X)\,\,=:\,\, {\mathcal Q}
\end{equation}
be the corresponding quotient map. Note that the above coherent analytic sheaf ${\mathcal Q}$ may
have torsion. For any point $y\, \in\, X$, the quotient ${\mathcal Q}/\mathbf{m}_y {\mathcal Q}$, where
$\mathbf{m}_y\, \subset\, {\mathcal O}_X$
is the maximal ideal associated to $y$, will be denoted by ${\mathcal Q}_y$.

The fiber over any point $y\,\in\, X$ of any vector bundle $W$ on $X$ will be denoted by $W_y$.

A holomorphic Lie algebroid connection on a holomorphic vector bundle $E$ on
$X$ --- for the Lie algebroid $(V,\, \phi)$ --- is a first order holomorphic differential operator
$$
D\,\,:\,\, E\,\,\longrightarrow\, \, E\otimes V^*
$$
such that
\begin{equation}\label{e-4}
D(fs) \,=\, fD(s) + s\otimes \phi^*(df)
\end{equation}
for all locally defined holomorphic sections $s$ of $E$ and all locally defined holomorphic
functions $f$ on $X$, where $\phi^*$ is the homomorphism constructed in \eqref{e2}.

Note that an usual holomorphic connection on $E$ is a holomorphic Lie algebroid connection on $E$ for the
holomorphic Lie algebroid $(TX,\, {\rm Id}_{TX})$; see \cite{At}, \cite{De} for holomorphic connections.

Consider a holomorphic Lie algebroid connection $D\,:\, E\,\longrightarrow\, E\otimes V^*$
on $E$ for the Lie algebroid $(V,\, \phi)$. Fix a point $x\, \in\, X$. Take any $v\, \in\, E_x$. Choose
a holomorphic section $s$ of $E$, defined on an open neighborhood $U$ of $x$, such that
$s(x)\,=\, v$. Consider $D(s)\, \in\, H^0\left(U,\, (E\otimes V^*)\big\vert_U\right)$. Let
\begin{equation}\label{ws}
\widehat{s} \,:=\, ({\rm Id}_E\times q)(D(s))(x) \, \in\, E_x\otimes {\mathcal Q}_x
\end{equation}
be the image, where $q$ is the projection in \eqref{e3}. If $s_1$ is another holomorphic section of
$E$, defined on an open neighborhood of $x$, with $s_1(x) \,=\, v$, then
\begin{equation}\label{e3a}
s-s_1\,\,=\,\, f\cdot t,
\end{equation}
where $t$ is a holomorphic section of $E$ defined on an
open neighborhood of $x$, and $f$ is a holomorphic function defined around $x$ with $f(x)\,=\, 0$. From \eqref{e-4}
and \eqref{e3a} it follows that
\begin{equation}\label{e3b}
(D(s) -D(s_1))(x) \,\,=\,\, t(x)\otimes \phi^*(df)(x) + f(x)\cdot (D(t))(x)
\,\,=\,\,t(x)\otimes \phi^*(df)(x),
\end{equation}
because $f(x)\,=\, 0$. Since $q(\phi^*(df))\,=\, 0$, where $q$ is the projection in \eqref{e3}, from
\eqref{e3b} it follows that $$\widehat{s} \,=\, ({\rm Id}_E\times q)(D(s))(x) \,=\,
({\rm Id}_E\times q)(D(s_1))(x).$$
Consequently, we get a linear map
\begin{equation}\label{e4}
{\mathcal S}_x\,\,:\,\, E_x\,\, \longrightarrow\,\, E_x\otimes {\mathcal Q}_x
\end{equation}
that sends any $v\, \in\, E_x$ to $\widehat{s}\,:=\, ({\rm Id}_E\times q)(D(s))(x)$, where $s$ is
any holomorphic section of $E$, defined around $x$, with $s(x)\,=\, v$. As shown above,
${\mathcal S}_x(v)$ does not depend on the choice of $s$.

\subsection{Parabolic vector bundles}\label{Para}

Let $S \,=\, \{x_1,\, \cdots,\, x_m\}\, \subset\, X$ be a nonempty
finite subset, whose elements will be called the parabolic points. Let $E$ be a
holomorphic vector bundle over $X$. A \emph{quasi-parabolic structure} on $E$ over
$x\, \in\, S$ is a strictly decreasing filtration of subspaces of the fiber $E_x$
\begin{equation}\label{eq:a1}
E_x \,=\, E^1_x \,\supsetneq\, E^2_x \,\supsetneq\, \, \cdots\, \supsetneq\, E^{\ell_x}_x\,
\supsetneq\, E^{\ell_x+1}_x \,=\, 0
\end{equation}
for some integer $\ell_x$. A parabolic structure on $E$ over $x\, \in\, S$ is a quasi-parabolic filtration
as in \eqref{eq:a1} together with a sequence of real numbers
\begin{equation}\label{ew}
0 \,\leq\, \alpha^x_1 \,< \,\cdots\, <\, \alpha^x_{\ell_x} \,< \,1.
\end{equation}
Note that $\ell_x\, \geq\, 1$ if $E\, \not=\, 0$. A parabolic structure on $E$ is parabolic structure
on $E$ on every $x\, \in\, S$ (see \cite{MS}, \cite{MY}). A parabolic vector bundle on $X$ consists
of a holomorphic vector bundle $E$ on $X$ together with a parabolic structure on $E$. For notational
convenience, such a parabolic vector bundle will also be denoted by $E_*$.

Take a parabolic vector bundle
\begin{equation}\label{e6}
E_*\,=\, (E,\, \{\{E^i_x\}_{i=1}^{\ell_x}\}_{x\in S},\, \{\{\alpha^x_i\}_{i=1}^{\ell_x}\}_{x\in S})
\end{equation}
as above. For any $x\,\in\, S$ and any $1\,\leq\, i\,\leq\,\ell_x+1$, let ${\mathcal E}_{x,i}$ be the unique holomorphic vector bundle
on $X$ that fits in the following short exact sequence of coherent analytic sheaves on $X$:
\begin{equation}\label{e5}
0\, \longrightarrow\, {\mathcal E}_{x,i}\, \xrightarrow{\,\,\,\iota_{x,i}\,\,}\, E \,
\longrightarrow\, E_x/E^i_x\, \longrightarrow\, 0.
\end{equation}
So ${\mathcal E}_{x,i}$ is identified with $E$ over the complement $X\setminus\{x\}$. These $\{{\mathcal E}_{x,i}\}$
form a decreasing sequence of subsheaves of $E$ for each $x\,\in \,S$:
\begin{equation}\label{eq:parSubsheaves}
E\,=\,{\mathcal E}_{x,1}\,\supsetneq\, {\mathcal E}_{x,2} \,\supsetneq\,\cdots
\,\supsetneq\, {\mathcal E}_{x,l_x} \,\supsetneq \,{\mathcal E}_{x,l_x+1}\,=\,E\otimes {\mathcal O}_X(-x).
\end{equation}

Consider the endomorphism bundle $\text{End}(E)\,=\, E\otimes E^*$. It has a coherent analytic
subsheaf
\begin{equation}\label{pe}
\text{End}_P(E_*)\, \subset \, \text{End}(E)
\end{equation}
defined by the following condition: For any
open subset $U\, \subset\, X$, any $$s\,\, \in\,\, H^0(U,\, \text{End}(E)\big\vert_U)$$ lies
in $H^0(U,\, \text{End}_P(E_*)\big\vert_U)$ if and only if 
$s(x)(E^i_x)\, \subset\, E^i_x$ for all $x\, \in\, S\bigcap U$ and $1\, \leq\, i\, \leq\, \ell_x$. Let
\begin{equation}\label{pe2}
\text{End}_n(E_*)\,\ \subset\, \ \text{End}_P(E_*)
\end{equation}
be the subsheaf defined by all $s$ as above such that $s(x)(E^i_x)\, \subset\, E^{i+1}_x$
for all $x\, \in\, S\bigcap U$ and $1\, \leq\, i\, \leq\, \ell_x$.

\begin{definition}\label{def2}
A parabolic vector bundle $E_*$ is called \textit{decomposable} if there are parabolic vector bundles
$A_*$ and $B_*$ of positive ranks such that $E_*\,=\, A_*\oplus B_*$ (see \cite{MY,Yok} for details on operations between parabolic vector bundles). A parabolic vector bundle $E_*$ is
called \textit{indecomposable} if it is not decomposable.
\end{definition}

\begin{lemma}\label{lem0}
A parabolic vector bundle $E_*$ on $X$ is indecomposable if and only if for every section $\tau
\, \in\, H^0(X,\, {\rm End}_P(E_*))$ there is some constant $\lambda\, \in\, {\mathbb C}$ such that
$\tau- \lambda\cdot {\rm Id}_{E}$ is a nilpotent endomorphism of $E_*$. Moreover, if such $\lambda$
exists, then it can be chosen as $\lambda=\frac{1}{\rm rank(E)}\rm trace(\tau)$.
\end{lemma}

\begin{proof}
If $E_*\,=\, A_*\oplus B_*$, where $A_*$ and $B_*$ are parabolic vector bundles of positive ranks,
then $\mu_1\cdot {\rm Id}_{A_*}\oplus \mu_2\cdot {\rm Id}_{B_*}\, \in\,
H^0(X,\, {\rm End}_P(E_*))$ is not of the above form if $\mu_1\, \not=\, \mu_2$ are
two distinct complex numbers.

To prove the converse, assume that $E_*$ is indecomposable. Take any $\tau \, \in\,
H^0(X,\, {\rm End}_P(E_*))$. 
The characteristic polynomial of $\tau$ expands as $\det (\tau -t\cdot{\rm Id}_{E})\,=\, \sum_{i=0}^r f_it^{r-i}$, where $r\,=\, \text{rank}(E)$ and $f_i=(-1)^i\operatorname{tr}(\wedge^i \tau)$. Then, $f_i$ are holomorphic functions
on $X$. Since $X$ is compact and connected, the functions $f_i$ are constants.
Thus the eigenvalues of $\tau (x)$ are independent of $x\, \in\, X$.

Let us prove that $\tau(x)$ can only have one eigenvalue. If $\tau(x)$ has more than one
eigenvalue, then $E$ splits into a direct sum
\begin{equation}\label{ej}
E\ =\ \bigoplus_i E_i
\end{equation}
given by the generalized eigenbundles
$E_i\, \subset\, E$ for $\tau$. Since $\tau$ is a parabolic endomorphism, for each parabolic point $x\,\in\, S$,
the endomorphism $\tau(x)$ preserves each subspace $E_x^j\,\subseteq\, E_x$. Thus, $E_x^j$ is itself
split as a direct sum of generalized eigenspaces of $\tau(x)\big\vert_{E_x^j}$
(see \eqref{ej}). It is clear by construction that
those generalized eigenspaces are $E_x^j \cap (E_i)_x$ and that
\begin{equation}
\label{eq:lem0}
E_x^j\ =\ \bigoplus_i E_x^j \cap (E_i)_x \quad \forall\ \, x\,\in\, S.
\end{equation}
Then, the filtration
$$(E_i)_x = E_x^1 \cap (E_i)_x \, \supset \, E_x^2 \cap (E_i)_x \supset \, \cdots \,\supset \, E_x^{l_x+1}
\cap (E_i)_x\, =\, 0$$
on each $x\,\in\, S$ induces a parabolic structure $E_{i,*}$ on $E_i$, where the weight
of $E_x^j\cap (E_i)_x$ is the maximum of all $\alpha_k^x$ such that $E_x^j\cap E_i
\, \subseteq\, E_x^{k}\cap E_i$. We note that this parabolic structure on $E_i$
coincides with the induced parabolic structure on the subbundle $E_i\subset E$ by $E_*$. Then \eqref{eq:lem0}
implies that $E_*\,=\,\bigoplus_i E_{i,*}$. But this contradicts the
given condition that $E_*$ is indecomposable.

Thus, $\tau$ has a unique eigenvalue $\lambda$ with multiplicity $r$. Since the trace of $\tau$ is the sum of
its eigenvalues, it follows that $\lambda\,=\,\frac{1}{r} \rm trace(\tau)$ and
$\tau- \lambda\cdot {\rm Id}_{E}$ is a nilpotent endomorphism of $E_*$.
\end{proof}

\section{Connections on parabolic bundles}\label{conn_para}

\subsection{Quasi-parabolic Lie algebroid connections}

Let $(V,\, \phi)$ be a holomorphic Lie algebroid on $X$. Take a parabolic vector bundle
$E_*$ over $X$ as in \eqref{e6}.

\begin{definition}
\label{def:parabolicConnection}
A \textit{quasi-parabolic Lie algebroid connection} on $E_*$ is a holomorphic Lie algebroid connection
$$
D\,\,:\,\, E\,\,\longrightarrow\, \, E\otimes V^*
$$
on $E$ (see \eqref{e-4}) such that
\begin{equation}\label{e7}
D({\mathcal E}_{x,i}) \,\, \subset\, \, {\mathcal E}_{x,i}\otimes V^*
\end{equation}
for all $x\, \in\, S$ and every $1\, \leq\, i\, \leq\, \ell_x+1$ (see \eqref{e5}).
\end{definition}

The following Lemmata from \cite{ABKS} will be useful in describing the
quasi-parabolic Lie algebroid connections.

\begin{lemma}[{\cite[Lemma 3.2]{ABKS}}]\label{lem1}
Fix a point $x\, \in\, S$ and $1\, \leq\, i\, \leq\, \ell_x+1$. Take a holomorphic Lie algebroid connection
$$D\ :\ E\ \longrightarrow \
E\otimes V^*$$ such that the homomorphism ${\mathcal S}_x \,\,:\,\, E_x\,\, \longrightarrow\,\, E_x\otimes {\mathcal Q}_x$ in \eqref{e4} maps $E^i_x\, \subset\, E_x$ to
$E^i_x\otimes {\mathcal Q}_x$. Then $D$ produces a homomorphism
$$
{\mathbb D}_{x,i}
 \,\, :\,\, {\mathcal E}_{x,i}\,\, \longrightarrow\,\, (E_x/E^i_x)\otimes\phi^*((K_X)_x)
\,\,\subset\,\, (E_x/E^i_x)\otimes V^*_x
$$
which satisfies the identity ${\mathbb D}_{x,i}(f\cdot s) \,=\, f(x)\cdot {\mathbb D}_{x,i}(s)$ for
every holomorphic section $s$ of ${\mathcal E}_{x,i}$ defined around $x\, \in\, S$ and every holomorphic
function $f$ defined on an open neighborhood of $x$, where $\phi^*$ is the homomorphism in \eqref{e2}.
\end{lemma}

\begin{proposition}[{\cite[Proposition 3.3]{ABKS}}]\label{prop1}
Take a holomorphic Lie algebroid connection $D\,:\, E\,\longrightarrow \, E\otimes V^*$ on $E$
for the holomorphic Lie algebroid $(V,\, \phi)$. It gives a quasi-parabolic
Lie algebroid connection on the
parabolic vector bundle $E_*$ if and only if the following two statements hold:
\begin{enumerate}
\item For every $x\, \in\, S$, the homomorphism ${\mathcal S}_x$ in \eqref{e4} maps $E^i_x\, \subset\, E_x$ to
$E^i_x\otimes {\mathcal Q}_x$ for all $1\, \leq\, i\, \leq\, \ell_x+1$.

\item For all $x\, \in\, S$, the homomorphism of fibers
$$\phi^*_x\ :\ (K_X)_x\ \longrightarrow\ V^*_x$$
(see \eqref{e2}) is the zero map.
\end{enumerate}
\end{proposition}

\subsection{Parabolic Lie algebroid connections}\label{para-conn}

Let $D$ be a quasi-parabolic Lie algebroid connection on a parabolic vector
bundle $E_*$. From Proposition \hyperref[prop1]{\ref*{prop1}.(2)} we know that the homomorphism
of fibers $\phi^*_x\, :\, (K_X)_x\, \longrightarrow\, V^*_x$ is the zero map for every
$x \,\in\, S$. Consequently, the homomorphism $\phi^* \,:\, K_X \,\longrightarrow\, V^*$
factors through the natural homomorphism $K_X \,\hookrightarrow\, K_X \otimes {\mathcal O}_X(S)$,
where $S\,=\,
\sum_{i=1}^m x_i$ is the reduced effective divisor; in other words, we have a unique homomorphism 
\begin{equation}\label{e12}
{\phi}^*_0\,\, : \,\, K_X \otimes {\mathcal O}_X(S) \,\,\longrightarrow\,\, V^*
\end{equation}
whose restriction to the subsheaf $K_X\, \subset\, K_X \otimes {\mathcal O}_X(S)$ coincides with $\phi^*$.

Note that for every $x\, \in\, S$, we have $(K_X \otimes {\mathcal O}_X(S))_x \,=\, \mathbb{C}$
(Poincar\'e adjunction
formula \cite[p.~146]{GH}). Consider the homomorphism $q(x)\, :\, V^*_x\, \longrightarrow\, {\mathcal Q}_x$ (see
\eqref{e3}). Composing it with ${\phi}^*_0(x)$ in \eqref{e12}, we have
\begin{equation}\label{e13}
\psi_x\, :=\, q(x)\circ ({\phi}^*_0(x))\,:\, \mathbb{C}\,=\,
(K_X \otimes {\mathcal O}_X(S))_x \, \longrightarrow\, {\mathcal Q}_x.
\end{equation}

{}From \eqref{e5} we have a short exact sequence
\begin{equation}\label{z0}
0\, \longrightarrow\, {\mathcal E}_{x,i+1}\, \longrightarrow\, {\mathcal E}_{x,i} \,
\longrightarrow\, E^i_x/E^{i+1}_x\,
\longrightarrow\, 0
\end{equation}
for all $x\, \in\, S$ and $1\, \leq\, i\, \leq\, \ell_x$. Since the quasi-parabolic Lie algebroid
connection $D$ preserves both ${\mathcal E}_{x,i}$ and ${\mathcal E}_{x,i+1}$ (see \eqref{e7}), from
\eqref{z0} we have a commutative diagram
\begin{equation}\label{cd2}
\begin{matrix}
0 & \longrightarrow & {\mathcal E}_{x,i+1} & \longrightarrow & {\mathcal E}_{x,i} & \longrightarrow & E^i_x/E^{i+1}_x &
\longrightarrow & 0\\
&&\,\,\, \Big\downarrow D &&\,\,\, \Big\downarrow D &&\,\,\, \Big\downarrow \widehat{D}\\
0 & \longrightarrow & {\mathcal E}_{x,i+1}\otimes V^* & \longrightarrow &
{\mathcal E}_{x,i}\otimes V^* & \longrightarrow & (E^i_x/E^{i+1}_x)\otimes V^* &\longrightarrow & 0
\end{matrix}
\end{equation}
where $\widehat{D}$ is induced by $D$. Since $D$ satisfies \eqref{e-4}, we have
\begin{equation}\label{z-1}
\widehat{D}(fs) \,=\, f\cdot\widehat{D}(s) + s\otimes \phi^*(df)
\end{equation}
for any section $s$ of $E^i_x/E^{i+1}_x$ (note that
$E^i_x/E^{i+1}_x$ is a torsion sheaf supported at $x$) and any holomorphic function $f$ defined around $x\, \in\, X$.
Consider the composition of maps
\begin{equation}\label{z1}
({\rm Id}_{E^i_x/E^{i+1}_x}\times q)\circ \widehat{D}\,\,:\,\, E^i_x/E^{i+1}_x \,\,
\longrightarrow\,\, (E^i_x/E^{i+1}_x)\otimes {\mathcal Q}_x,
\end{equation}
where $q$ is the homomorphism in \eqref{e3}. From \eqref{z-1} it follows that
$({\rm Id}_{E^i_x/E^{i+1}_x}\times q)\circ \widehat{D}$
has the property
$$\left ( ({\rm Id}_{E^i_x/E^{i+1}_x}\times q)\circ \widehat{D} \right)(fs) \,\,=\,\, \left(f\cdot
({\rm Id}_{E^i_x/E^{i+1}_x}\times q)\circ\widehat{D}\right)(s),$$
and hence
$({\rm Id}_{E^i_x/E^{i+1}_x}\times q)\circ \widehat{D}$
 produces a $\mathbb C$--linear homomorphism
\begin{equation}\label{z2}
\widehat{D}_{x,i}\,\,:\,\, E^i_x/E^{i+1}_x \,\, \longrightarrow\,\, (E^i_x/E^{i+1}_x)\otimes {\mathcal Q}_x
\end{equation}
of fibers.

\begin{definition}\label{def:paraLieconn}
The quasi-parabolic Lie algebroid connection $D$ on the parabolic vector bundle $E_*$ for the Lie
algebroid $(V,\, \phi)$ will be called a \textit{parabolic Lie algebroid connection} if the homomorphism
$\widehat{D}_{x,i}$ in \eqref{z2} coincides with the homomorphism $E^i_x/E^{i+1}_x\,\, \longrightarrow\,
\, (E^i_x/E^{i+1}_x)\otimes {\mathcal Q}_x$ defined by $v\, \longmapsto\, v\otimes (\psi_x(\alpha^x_i))$,
where $\psi_x$ is the homomorphism in \eqref{e13} and $\alpha^x_i$ is the parabolic weight in \eqref{ew}.
\end{definition}

Notice that, for $(V,\, \phi)\, = \, (TX\otimes {\mathcal O}_X(-S),\, \iota)$, where $\iota\, : \, TX\otimes 
\mathcal{O}_X(-S) \longrightarrow TX$ is the natural inclusion map, the notion of
parabolic Lie algebroid connection coincides with the usual notion 
of parabolic connection. We also note that when $(V,\, \phi)\, = \, (TX\otimes \mathcal{O}_X(-S), \, 0)$,
a $(V,\, \phi)$-Lie algebroid connection is a strongly parabolic Higgs bundle.

\subsection{The Atiyah exact sequence} 

Let $(V, \, \phi)$ be a holomorphic Lie algebroid on $X$. We first recall the notion of generalized Lie 
algebroid connections on a vector bundle over an open subset of $X$ defined in \cite{ABKS}.

Take an open subset $U\, \subset\, X$, and denote $S_U\, :=\, U \bigcap S$. Fix
a holomorphic function $w$ on $U$. The restriction $(V,\, \phi)\big\vert_U$ of the
holomorphic Lie algebroid $(V,\, \phi)$ to the open subset $U$ will be denoted by $(V_U,\,
\phi_U)$. A {\it generalized Lie algebroid connection with weight $w$} on a holomorphic
vector bundle ${\mathcal W}$ defined over $U$ is a holomorphic differential operator 
$$D \,\,:\,\, {\mathcal W} \,\,\longrightarrow \,\, {\mathcal W} \otimes V^*_U $$
satisfying the Leibniz rule
\begin{equation}\label{e15}
D(fs) \,\,=\,\, f D(s)\, +\, w \cdot s \otimes \phi^*_U (df),
\end{equation}
where $f$ is any locally defined holomorphic function on $U$ and $s$ is any locally
defined holomorphic section of $\mathcal W$; as before, $\phi^*_U$ is the dual of the
anchor map $\phi_U$. Note that $D$ is a holomorphic Lie algebroid connection on
$\mathcal W$ if $w$ is the constant function $1$ on $U$.

A generalized Lie algebroid connection on a holomorphic vector bundle 
$\mathcal W$ defined over $U$ is a pair $(w,\, D)$, where $w$ is a holomorphic function
on $U$ and $D$ is a generalized Lie algebroid connection on $\mathcal W$ with weight $w$. 

Take a generalized Lie algebroid connection $(w,\, D)$ on ${\mathcal W}\, \longrightarrow\, U$. Consider
the homomorphism
$$
({\rm Id}_{\mathcal W}\otimes q)\circ D\,\, :\,\, {\mathcal W}\,\, \longrightarrow\,\, {\mathcal W}\otimes
{\mathcal Q}\big\vert_U,
$$
where $q$ is the projection in \eqref{e3}. From \eqref{e15} it follows immediately that
$$\left(({\rm Id}_{\mathcal W}\otimes q)\circ D \right)(f\cdot s)\,=\, f\cdot \left( ({\rm Id}_{\mathcal W} \otimes q) \circ D \right) (s),$$
where $s$ is any locally defined 
holomorphic section of $\mathcal W$ and $f$ is any locally defined holomorphic function on $U$. Therefore,
$({\rm Id}_{\mathcal W}\otimes q)\circ D$ produces a $\mathbb C$--linear homomorphism of fibers
\begin{equation}\label{e16}
{\mathcal S}_x\,\,:\,\, {\mathcal W}_x\, \,\longrightarrow\, \,{\mathcal W}_x\otimes {\mathcal Q}_x
\end{equation}
for every point $x \,\in\, U$.

Take a holomorphic vector bundle $E$ over $X$.
Let $\mathcal{C}_{E,V}$ denote the sheaf of generalized Lie algebroid connections on $E$ over $X$.
So the sections of $\mathcal{C}_{E,V}$ over an open set $U \,\subset\, X$ are the
generalized Lie algebroid connections on $E\big\vert_U$.
It is evident that $\mathcal{C}_{E,V}$ is a locally free coherent analytic sheaf fitting in the
short exact sequence
\begin{equation}\label{y1}
0\, \longrightarrow\, \text{End}(E)\otimes V^* \, \longrightarrow\, \mathcal{C}_{E,V}
\, \stackrel{\sigma}{\longrightarrow}\, {\mathcal O}_X \, \longrightarrow\, 0;
\end{equation}
the above projection $\sigma$ sends any locally defined generalized Lie algebroid connection $(w,\, D)$
on $E\big\vert_U$ to the holomorphic function $w$ on $U$.

Let $E_*\,=\, (E,\, \{\{E^i_x\}_{i=1}^{\ell_x}\}_{x\in S},\, \{\{\alpha^x_i\}_{i=1}^{\ell_x}\}_{x\in S})$
be a parabolic vector bundle on $X$. Take a holomorphic Lie algebroid $(V, \, \phi)$ on $X$ satisfying the
following condition: For every $x\, \in\, S$, the homomorphism of fibers $\phi^*_x\, :\, (K_X)_x\, \longrightarrow\,
V^*_x$ (see \eqref{e2}) is the zero map.

Take an open subset $U\, \subset\, X$. A \textit{generalized quasi-parabolic Lie algebroid connection} on $E_*$
over $U$ is a generalized Lie algebroid connection
$$
D\,\,:\,\, E\big\vert_U \,\,\longrightarrow\, \, (E\otimes V^*)\big\vert_U
$$
on $E\big\vert_U$ (see \eqref{e-4}) such that
\begin{equation}\label{q1}
D({\mathcal E}_{x,i}\big\vert_U) \,\, \subset\, \, ({\mathcal E}_{x,i}\otimes V^*)\big\vert_U
\end{equation}
for all $x\, \in\, S\bigcap U$ and every $1\, \leq\, i\, \leq\, \ell_x+1$ (see \eqref{e5}).

The following is a generalization of Proposition \ref{prop1}.

\begin{lemma}[{\cite[Lemma 3.6]{ABKS}}]\label{lem3}
A generalized Lie algebroid connection $(w,\, D)$ on $E\big\vert_U$
gives a generalized quasi-parabolic Lie algebroid connection on the
parabolic vector bundle $E_*\big\vert_U$ if and only if the following holds:
For every $x\, \in\, S\bigcap U$, the homomorphism
$$
{\mathcal S}_x\,\,:\,\, E_x\, \,\longrightarrow\, \, E_x\otimes {\mathcal Q}_x
$$
in \eqref{e16} maps $E^i_x\, \subset\, E_x$ to
$E^i_x\otimes {\mathcal Q}_x$ for all $1\, \leq\, i\, \leq\, \ell_x+1$.
\end{lemma}

Let $(w,\, D)$ be a generalized quasi-parabolic Lie algebroid connection on $E_*\big\vert_U$.
Then using the analog of \eqref{cd2} it follows that the analog of the
homomorphisms ${\mathcal S}_x\big\vert_{E^i_x}$,
$1\,\leq\, i\, \leq\, \ell_x$ (see Lemma \ref{lem3}), produce homomorphisms
\begin{equation}\label{z3}
\widehat{D}_{x,i}\,\,:\,\, E^i_x/E^{i+1}_x \,\, \longrightarrow\,\, (E^i_x/E^{i+1}_x)\otimes {\mathcal Q}_x
\end{equation}
(this is similar to \eqref{z2}).

\begin{definition}
\label{def:genParabolicCon}
A \textit{generalized parabolic Lie algebroid connection} on $E_*\big\vert_U$ is
a generalized quasi-parabolic Lie algebroid connection $(w,\, D)$ on $E_*\big\vert_U$ such that
for every $x\, \in\, S\bigcap U$, and every $1\,\leq\, i\, \leq\, \ell_x$, the homomorphism
$\widehat{D}_{x,i}$ in \eqref{z3} coincides with the homomorphism
$$E^i_x/E^{i+1}_x \,\, \longrightarrow\,\, (E^i_x/E^{i+1}_x)\otimes {\mathcal Q}_x$$ defined by
$v\, \longmapsto\, w(x)\cdot v\otimes (\psi_x(\alpha^x_i))$, where $\psi_x$ is the homomorphism
in \eqref{e13} and $\alpha^x_i$ is the parabolic weight in \eqref{ew}.
\end{definition}

The sheaf of generalized parabolic Lie algebroid connections on $E_*$ is evidently a subsheaf
of $\mathcal{C}_{E,V}$ (see \eqref{y1}). In fact, it is a coherent analytic subsheaf
of $\mathcal{C}_{E,V}$.

The sheaf of generalized parabolic Lie algebroid connections on $E_*$
will be denoted by $\mathcal{C}_{E_*,V}$.

\begin{lemma}[{\cite[Lemma 3.7]{ABKS}}]\label{lem4}
The sheaf of generalized parabolic Lie algebroid connections on $E_*$, namely $\mathcal{C}_{E_*,V}$,
fits in the following short exact sequence of holomorphic vector bundles on $X$:
\begin{equation}\label{Atiya_para_Lie}
0\, \longrightarrow\, {\rm End}_n(E_*)\otimes V^* \, \longrightarrow\, \mathcal{C}_{E_*,V}
\, \stackrel{\sigma}{\longrightarrow}\, {\mathcal O}_X \, \longrightarrow\, 0,
\end{equation}
where ${\rm End}_n(E_*)$ is defined in \eqref{pe2}.
\end{lemma}

{}From the construction of $\mathcal{C}_{E_*,V}$ it follows immediately that a parabolic Lie algebroid
connection on $E_*\big\vert_U$ (see Definition \ref{def:paraLieconn}) is precisely a holomorphic
splitting of the short exact sequence in Lemma \ref{lem4} over $U\, \subset\, X$. More precisely, if
$$
\tau\, \, :\,\, {\mathcal O}_U\,\, \longrightarrow\,\, \mathcal{C}_{E_*,V}
$$
is a holomorphic splitting of the short exact sequence in Lemma \ref{lem4} over $U\,
\subset\, X$ (so $\sigma\circ\tau\,=\, {\rm Id}_{{\mathcal O}_U}$), then $\tau(1)$ is a parabolic Lie
algebroid connection on $E_*\big\vert_U$.
Conversely, if $D$ is a parabolic Lie algebroid connection on $E_*\big\vert_U$, then
we have a holomorphic splitting
$$
\tau'\, \, :\,\, {\mathcal O}_U\,\, \longrightarrow\,\, \mathcal{C}_{E_*,V}
$$
of the short exact sequence in Lemma \ref{lem4} over $U\, \subset\, X$ defined by
$\tau'(f)\,=\, f\cdot D$, where $f$ is any holomorphic function on some open subset of $U$.

The short exact sequence in Lemma \ref{lem4} will be called the \textit{Atiyah exact sequence}
for the parabolic vector bundle $E_*$ for the holomorphic Lie algebroid $(V, \, \phi)$.

\section{Connecting Atiyah exact sequences}

Take a holomorphic Lie algebroid $(V,\, \phi)$ on $X$. A necessary condition for a parabolic
vector bundle $E_*$ on
$X$ to admit a quasi-parabolic Lie algebroid connection for $(V,\, \phi)$ is the following:
For all $x\, \in\, S$, the homomorphism of fibers 
$\phi^*_x\, :\, (K_X)_x\, \longrightarrow\, V^*_x$ (see \eqref{e2}) is the zero map (see Proposition
\ref{prop1}). So if $E_*$ admits a quasi-parabolic Lie algebroid connection for $(V,\, \phi)$, then
the condition
\begin{equation}\label{e8}
\phi_x\ \, =\, \ 0
\end{equation}
holds for all $x\, \in\, S$.

Assume that \eqref{e8} holds for $(V,\, \phi)$. Consequently, the anchor map
$\phi\, :\, V\, \longrightarrow\,TX$ uniquely factors as
$$
V \ \longrightarrow\ TX\otimes {\mathcal O}_X(-S)\ \hookrightarrow\ TX.
$$
Let
\begin{equation}\label{e9}
\phi_0\ :\ V \ \longrightarrow\ TX\otimes {\mathcal O}_X(-S)
\end{equation}
be the above map. Note that $\phi_0$ is the dual of the homomorphism in \eqref{e12}.

In this section we will compare the Atiyah exact sequence in Lemma \ref{lem4} for $(V,\, \phi)$
with the corresponding Atiyah exact sequence for the holomorphic Lie algebroid
$(TX\otimes {\mathcal O}_X(-S),\, \iota)$, where
\begin{equation}\label{io}
\iota\ :\ TX\otimes {\mathcal O}_X(-S)\ \hookrightarrow\ TX
\end{equation}
is the natural inclusion map.

For notational convenience, the holomorphic line bundle $TX\otimes {\mathcal O}_X(-S)$ will be
denoted by $\mathbb T$.

Note that this holomorphic Lie algebroid
$({\mathbb T},\, \iota)$ satisfies the condition in \eqref{e8}.

Setting $(V,\, \phi)\ =\ (TX\otimes {\mathcal O}_X(-S),\, \iota)\ =\
(\mathbb T,\, \iota)$ in Lemma \ref{lem4}, we get a
short exact sequence of holomorphic vector bundles on $X$
\begin{equation}\label{k1}
0\, \longrightarrow\, {\rm End}_n(E_*)\otimes K_X\otimes{\mathcal O}_X(S) \, \longrightarrow\,
\mathcal{C}_{E_*,TX\otimes {\mathcal O}_X(-S)}\, =\, \mathcal{C}_{E_*,{\mathbb T}}
\, \stackrel{\sigma^0}{\longrightarrow}\, {\mathcal O}_X \, \longrightarrow\, 0
\end{equation}
(see \eqref{Atiya_para_Lie}). So, a holomorphic section $\delta$ of $\mathcal{C}_{E_*,{\mathbb T}}$ over
an open subset $U\, \subset\, X$ is a holomorphic differential operator
$$
\delta\ :\ E\big\vert_U \ \longrightarrow\ (E\otimes K_X\otimes {\mathcal O}_X(S))\big\vert_U\,=\,
(E\otimes {\mathbb T}^*)\big\vert_U
$$
satisfying the conditions in Definition \ref{def:genParabolicCon}. Consider the composition of homomorphisms
$$
(\text{Id}_E\otimes \phi^*_0)\big\vert_U\circ \delta\ :\ E\big\vert_U \ \longrightarrow\ 
(E\otimes V^*)\big\vert_U,
$$
where $\phi^*_0$ is the dual of the homomorphism $\phi_0$ in \eqref{e9}. It is straightforward to
check that
$$
(\text{Id}_E\otimes \phi^*_0)\big\vert_U\circ \delta\ \in\ H^0(U,\, \mathcal{C}_{E_*,V}),
$$
where $\mathcal{C}_{E_*,V}$ is the vector bundle in \eqref{Atiya_para_Lie}.
Consequently, we get an ${\mathcal O}_X$--linear homomorphism
\begin{equation}\label{k2}
\Phi\ :\ \mathcal{C}_{E_*,{\mathbb T}}\ \longrightarrow\ \mathcal{C}_{E_*,V}
\end{equation}
that sends any holomorphic section $\delta$ of $\mathcal{C}_{E_*,{\mathbb T}}$ over $U\, \subset\, X$ to
the section $(\text{Id}_E\otimes \phi^*_0)\big\vert_U\circ
\delta$ of $\mathcal{C}_{E_*,V}\big\vert_U$ constructed above.

The homomorphism $\Phi$ in \eqref{k2} fits into the following commutative diagram of homomorphisms
\begin{equation}\label{k3}
\begin{matrix}
0 & \longrightarrow & {\rm End}_n(E_*)\otimes K_X\otimes{\mathcal O}_X(S)& \longrightarrow
& \mathcal{C}_{E_*,{\mathbb T}} & \stackrel{\sigma^0}{\longrightarrow} & {\mathcal O}_X
& \longrightarrow & 0\\
&& 
{\rm Id}_{{\rm End}_n(E_*)}\otimes \phi^*_0\Big\downarrow\,\, \text{ }\text{ }\quad\, \qquad
\, \quad && \Phi\Big\downarrow\,\,\,\, && \Vert\\
0 & \longrightarrow & {\rm End}_n(E_*)\otimes V^* & \longrightarrow & \mathcal{C}_{E_*,V}
& \stackrel{\sigma}{\longrightarrow} & {\mathcal O}_X & \longrightarrow & 0
\end{matrix}
\end{equation}
where the top and bottom exact sequences are as in \eqref{k1} and \eqref{Atiya_para_Lie}
respectively and $\phi^*_0$ is the dual of the homomorphism $\phi_0$ in \eqref{e9}.

Consider the bottom exact sequence of holomorphic vector bundles in \eqref{k3}. The
obstruction to its holomorphic splitting is given by a cohomology class
\begin{equation}\label{ze}
\zeta\,\, \in\, \, H^1(X,\, \text{Hom}({\mathcal O}_X,\, {\rm End}_n(E_*)\otimes V^*))
\,\,=\,\, H^1(X,\, {\rm End}_n(E_*)\otimes V^*).
\end{equation}
Similarly, let
\begin{equation}\label{ze2}
\zeta_0\,\, \in\, \, H^1(X,\, {\rm End}_n(E_*)\otimes K_X\otimes{\mathcal O}_X(S))
\end{equation}
be the cohomology class corresponding to the top exact sequence in \eqref{k3}. From the commutative
diagram in \eqref{k3} it follows immediately that
\begin{equation}\label{ze3}
\zeta\ =\ ({\rm Id}_{{\rm End}_n(E_*)}\otimes \phi^*_0)_* (\zeta_0),
\end{equation}
where $\zeta$ (respectively, $\zeta_0$) is the cohomology class in \eqref{ze} (respectively, \eqref{ze2}),
and $$({\rm Id}_{{\rm End}_n(E_*)}\otimes \phi^*_0)_*\ :\ H^1(X,\, {\rm End}_n(E_*)\otimes K_X\otimes
{\mathcal O}_X(S))\ \longrightarrow\ H^1(X,\, {\rm End}_n(E_*)\otimes V^*)$$ is the homomorphism of
cohomologies induced by the homomorphism ${\rm Id}_{{\rm End}_n(E_*)}\otimes \phi^*_0$ in \eqref{k3}.

\section{Logarithmically non-split Lie algebroids and connections}

\begin{definition}\label{def-1}
Let $(V,\, \phi)$ be a holomorphic Lie algebroid on $X$ such that $\phi_x\, =\, 0$
for all $x\, \in\, S$ (see \eqref{e8}). It is
called \textit{logarithmically split} if there is a holomorphic ${\mathcal O}_X$--linear homomorphism
\begin{equation}\label{s}
\sigma\ :\ TX\otimes{\mathcal O}_X(-S)\ \longrightarrow\ V
\end{equation}
such that $\phi_0\circ\sigma\,=\, {\rm Id}_{TX\otimes{\mathcal O}_X(-S)}$, where $\phi_0$
is the homomorphism in \eqref{e9} given by $\phi$. A holomorphic Lie algebroid
$(V,\, \phi)$ is called \textit{logarithmically non-split} if it is not logarithmically split.
\end{definition}

\begin{theorem}\label{thm1}
Let $(V,\, \phi)$ be a holomorphic Lie algebroid on $X$ such that $\phi_x\, =\, 0$
for all $x\, \in\, S$ (as in \eqref{e8}), and $(V,\,\phi)$ is logarithmically non-split (see Definition
\ref{def-1}). Then any indecomposable parabolic vector bundle $E_*$ on $X$ (see Definition
\ref{def2}), with parabolic divisor $S$, admits a parabolic Lie algebroid connection for $(V,\,\phi)$.
\end{theorem}

\begin{proof}
To prove the theorem it suffices to show that the cohomology class $\zeta$ in \eqref{ze} vanishes.

Consider ${\rm End}_P(E_*)$ and $\text{End}_n(E_*)$ defined in \eqref{pe} and \eqref{pe2} respectively.
It can be shown (see, for instance, \cite{Yok}) that
\begin{equation}\label{k4}
{\rm End}_n(E_*)^*\ =\ {\rm End}_P(E_*)\otimes {\mathcal O}_X(S).
\end{equation}
To see this, consider the natural nondegenerate pairing $\text{End}(E)\otimes\text{End}(E)\,
\longrightarrow\, {\mathcal O}_X$ defined by $A\otimes B\, \longmapsto\, \text{trace}(AB)$. It
restricts to a nondegenerate pairing $$\text{End}_n(E_*)\otimes\text{End}_P(E_*)\
\longrightarrow\ {\mathcal O}_X(-S).$$ From this \eqref{k4} follows immediately.

Consider $\zeta$ and $\zeta_0$ constructed in \eqref{ze} and \eqref{ze2} respectively. 
In view of \eqref{k4}, using Serre duality,
\begin{equation}\label{k5}
\zeta\,\, \in\, \, H^0(X,\, {\rm End}_P(E_*)\otimes{\mathcal O}_X(S) \otimes V\otimes K_X)^*,
\end{equation}
\begin{equation}\label{k6}
\zeta_0\,\, \in\, \, H^0(X,\, {\rm End}_P(E_*))^*.
\end{equation}
The homomorphism
\begin{equation}\label{k5a}
H^0(X,\, {\rm End}_P(E_*)\otimes{\mathcal O}_X(S) \otimes V\otimes K_X)\ \longrightarrow\
{\mathbb C}
\end{equation}
given by $\zeta$ in \eqref{k5} will also be denoted by $\zeta$. Similarly, the homomorphism
\begin{equation}\label{k6a}
H^0(X,\, {\rm End}_P(E_*))\ \longrightarrow\ {\mathbb C}
\end{equation}
given by $\zeta_0$ in \eqref{k6} will also be denoted by $\zeta_0$. 

Consider the homomorphism $\phi_0$ in \eqref{e9}. Note that
$$
\phi_0\ \in\ H^0(X,\, \text{Hom}({\mathcal O}_X(S)\otimes V\otimes K_X,\, {\mathcal O}_X)).
$$
Thus, it produces a homomorphism
\begin{equation}\label{hs}
\phi'\ :=\ {\rm Id}_{{\rm End}_P(E_*)}\otimes\phi_0\ :\ {\rm End}_P(E_*)
\otimes{\mathcal O}_X(S) \otimes V\otimes K_X\ \longrightarrow\ {\rm End}_P(E_*);
\end{equation}
so $\phi'$ sends any $a\otimes b\, \in\, ({\rm End}_P(E_*)\otimes{\mathcal O}_X(S) \otimes V\otimes K_X)_x$,
where $x\, \in\, X$, $a\, \in\, {\rm End}_P(E_*)_x$ and $b\, \in\, ({\mathcal O}_X(S) \otimes V\otimes
K_X)_x$, to $(\phi_0)_x(b)\cdot a$. Let
\begin{equation}\label{k7}
\phi'_*\ :\ H^0(X,\, {\rm End}_P(E_*)\otimes{\mathcal O}_X(S) \otimes V\otimes K_X)\ \longrightarrow\
H^0(X,\, {\rm End}_P(E_*))
\end{equation}
be the homomorphism of cohomologies induced by the homomorphism $\phi'$ in \eqref{hs}. An
immediate consequence of \eqref{ze3} is that the following diagram is commutative: 
\begin{equation}\label{k8}
\begin{matrix}
H^0(X,\, {\rm End}_P(E_*)\otimes{\mathcal O}_X(S) \otimes V\otimes K_X)& \xrightarrow{\,\,\,\zeta\,\,\,}
& {\mathbb C}\\
\phi'_*\Big\downarrow\,\,\,\,\, && \Vert\\
H^0(X,\, {\rm End}_P(E_*)) & \xrightarrow{\,\,\,\zeta_0\,\,\,} & {\mathbb C}
\end{matrix}
\end{equation}
where $\phi'_*$, $\zeta_0$ and $\zeta$ are as in \eqref{k7}, \eqref{k6} and \eqref{k5} respectively
(see also \eqref{k6a} and \eqref{k5a}).

Now we will analyze the homomorphism $\zeta_0$ in \eqref{k8}. Take any
\begin{equation}\label{k9a}
\tau\ \in\ H^0(X,\, {\rm End}_P(E_*))\ \subset\ H^0(X,\, {\rm End}(E))
\end{equation}
(see \eqref{pe}). Since we assumed that $E_*$ is indecomposable, from Lemma \ref{lem0} we know that the endomorphism
\begin{equation}\label{k9}
\tau_0\ :=\ \tau- \frac{1}{\rm rank(E)}{\rm trace(\tau)} \cdot {\rm Id}_{E_*}
\end{equation}
is actually a nilpotent endomorphism of $E$. We will show that
\begin{equation}\label{k10}
\zeta_0(\tau_0)\ =\ 0,
\end{equation}
where $\zeta_0$ is the homomorphism in \eqref{k8}.

To prove \eqref{k10}, for $1\, \leq\, i\, \leq\, r\,=\,{\rm rank}(E)$, let $E_i\, \subset\, E$ be the
subbundle of $E$
generated by $\text{kernel}(\tau^i_0)$. In other words, $E_i$ is the inverse image of the torsion part
of $E/\tau^i_0(E)$ under the quotient map $E\, \longrightarrow\, E/\tau^i_0(E)$. So we have a filtration
of subbundles of $E$
\begin{equation}\label{k11}
0\,=\, E_r\, = \, \cdots \, =\, E_{k}\, \subsetneq\, E_{k-1}\,
\subsetneq \, \cdots\, \subsetneq\, E_1\, \subsetneq\, E_0\, :=\, E
\end{equation}
for some $0<k\le r$. The homomorphism $\tau_0$ is evidently nilpotent with respect to the filtration in \eqref{k11}, meaning
\begin{equation}\label{k11a}
\tau_0(E_j)\ \subset\ E_{j+1}
\end{equation}
for all $0\, \leq\, j\, \leq\, r-1$.
Note that $\tau_0\, \in\, H^0(X,\, {\rm End}_P(E_*))$ (see \eqref{k9a} and \eqref{k9}).

Consider $\text{End}_n(E_*)$ in \eqref{pe2}. Let
\begin{equation}\label{k12}
\text{End}^f_n(E_*)\,\ \subset\,\ \text{End}_n(E_*)
\end{equation}
be the subbundle given by the sheaf of endomorphisms of $E$ that lies in
$\text{End}_n(E_*)$ and also preserves the filtration in \eqref{k11}. From
the fact that $\tau_0\, \in\, H^0(X,\, {\rm End}_P(E_*))$ it follows that $\text{End}^f_n(E_*)$ is
actually a subbundle of $\text{End}_n(E_*)$.

Consider the vector bundle $\mathcal{C}_{E_*,{\mathbb T}}$ in \eqref{k1} given by the sheaf of generalized 
parabolic Lie algebroid connections on $E_*$. Let
$$
\mathcal{C}^f_{E_*,{\mathbb T}}\,\ \subset\, \ \mathcal{C}_{E_*,{\mathbb T}}
$$
be the subbundle given by the sheaf of generalized parabolic Lie algebroid connections that preserve
the filtration in \eqref{k11}. Again from the fact that $\tau_0\, \in\, H^0(X,\, {\rm End}_P(E_*))$
it follows that $\mathcal{C}^f_{E_*,{\mathbb T}}$ is a subbundle of $\mathcal{C}_{E_*,{\mathbb T}}$.
So from \eqref{k1} we have the short exact sequence of holomorphic vector bundles on $X$
$$
0\, \longrightarrow\, {\rm End}^f_n(E_*)\otimes K_X\otimes{\mathcal O}_X(S) \, \longrightarrow\,
\mathcal{C}^f_{E_*,{\mathbb T}}
\, \longrightarrow\, {\mathcal O}_X \, \longrightarrow\, 0
$$
which fits in the following commutative diagram
\begin{equation}\label{k13}
\begin{matrix}
0 & \longrightarrow & {\rm End}^f_n(E_*)\otimes K_X\otimes{\mathcal O}_X(S)& \longrightarrow &
\mathcal{C}^f_{E_*,{\mathbb T}}& \longrightarrow & {\mathcal O}_X & \longrightarrow & 0\\
&&\,\,\, \Big\downarrow\iota^0 &&\Big\downarrow &&\,\,\,\Big\downarrow\iota'\\
0 & \longrightarrow & {\rm End}_n(E_*)\otimes K_X\otimes{\mathcal O}_X(S)& \longrightarrow &
\mathcal{C}_{E_*,{\mathbb T}}& \longrightarrow & {\mathcal O}_X & \longrightarrow & 0
\end{matrix}
\end{equation}
where $\iota^0$ and $\iota'$ are the natural inclusion maps. Let
\begin{equation}\label{k14}
\zeta^f_0\,\, \in\, \, H^1(X,\, \text{Hom}({\mathcal O}_X,\, {\rm End}^f_n(E_*)\otimes K_X\otimes
{\mathcal O}_X(S)))\,\,=\,\, H^1(X,\, {\rm End}^f_n(E_*)\otimes K_X\otimes{\mathcal O}_X(S))
\end{equation}
be the cohomology class corresponding to the top exact sequence in \eqref{k13}. From \eqref{k13}
it follows immediately that
\begin{equation}\label{k15}
\zeta_0\,\ =\, \ \iota^0_*(\zeta^f_0),
\end{equation}
where $\zeta_0$ and $\zeta^f_0$ are as in \eqref{ze2} and \eqref{k14} respectively, and
$$
\iota^0_*\ :\ H^1(X,\, {\rm End}^f_n(E_*)\otimes K_X\otimes{\mathcal O}_X(S))\
\longrightarrow\ H^1(X,\, {\rm End}_n(E_*)\otimes K_X\otimes{\mathcal O}_X(S))
$$
is the homomorphism of cohomologies given by $\iota^0$ in \eqref{k13}.

{}From \eqref{k11a} and \eqref{k15} it follows that \eqref{k10} holds. This derivation uses
the fact that for endomorphisms $A,\, B\,\in\, {\rm End}({\mathbb C}^r)$, such that $A$
preserves a filtration of ${\mathbb C}^r$ and $B$ is nilpotent with respect to that filtration,
we have ${\rm trace}(A\circ B)\,=\, 0$, because $A\circ B$ is nilpotent.

Now, by linearity of $\zeta_0$, \eqref{k9} and \eqref{k10}, we have
$$\zeta_0(\tau)\, =\, \zeta_0(\tau_0)+\frac{1}{\rm rank(E)} {\rm trace(\tau)}\zeta_0(\rm Id_{E_*}) = \frac{\zeta_0(\rm Id_{E_*})}{\rm rank(E)} {\rm trace(\tau)}$$
for all $\tau\ \in\ H^0(X,\, {\rm End}_P(E_*))$. Set $\beta\, := \, \frac{\zeta_0(\rm Id_{E_*})}{\rm rank(E)}
\,\in\, \mathbb{C}$. Then $\zeta_0=\beta \cdot \rm trace$. Thus, from \ref{k8},
\begin{equation} \label{k15c}
\zeta\ =\ \beta \cdot \rm trace \circ \phi'_*.
\end{equation}
In view of this, and the fact that the trace and the contraction with $\phi_0$ evidently commute, the
following diagram is commutative:
\begin{equation}\label{k15d}
\begin{matrix}
H^0(X,\, {\rm End}_P(E_*)\otimes{\mathcal O}_X(S) \otimes V\otimes K_X)& \xrightarrow{\,\,\,\,\,\zeta\,\,\,\,\,}
& {\mathbb C}\\
\rm trace\Big\downarrow\,\,\,\,\, && \,\,\,\, \Big\uparrow \beta \\
H^0(X,\, {\mathcal O}_X(S) \otimes V\otimes K_X) & \xrightarrow{\,\,\,(\phi_0)_*\,\,\,} & {\mathbb C}
\end{matrix}
\end{equation}
where $(\phi_0)_*$ is the contraction map of elements in $H^0(X,\, {\mathcal O}_X(S) \otimes V\otimes K_X)$ with
$\phi_0\,\in\, H^0(X,\, \text{Hom}(V,\, TX\otimes{\mathcal O}_X(-S))) \, = \, H^0(X,\, ({\mathcal O}_X(S) \otimes V\otimes K_X)^*)$.

To prove the theorem by contradiction, assume that the homomorphism $\zeta$ in \eqref{k8} is nonzero. Then by \eqref{k15d}, the map $(\phi_0)_*$ must be nonzero, so there must exist a section
$$
\varpi\,\ \in\,\ H^0(X,\, {\mathcal O}_X(S) \otimes V\otimes K_X)
$$
such that $(\phi_0)_*(\varpi) \ \, \not=\,\ 0$. Note that $\varpi$ is also an homomorphism $\varpi\, :\, TX\otimes {\mathcal O}_X(-S)\,
\longrightarrow\, V$, and considering $\varpi$ as such,
\begin{equation}\label{k21}
\phi_0\circ\varpi\ =\ (\phi_0)_* (\varpi)\cdot {\rm Id}_{TX\otimes {\mathcal O}_X(-S)},
\end{equation}
where $\phi_0$ is the homomorphism in \eqref{e9}. Now from \eqref{k21} and the fact that $(\phi_0)_* (\varpi)\ne 0$ it follows that the homomorphism
$$
\sigma\ :=\ \frac{1}{(\phi_0)_* (\varpi)}\cdot \varpi \ :\ TX\otimes {\mathcal O}_X(-S)\
\longrightarrow\ V$$
satisfies the condition $\phi_0\circ\sigma\,=\, {\rm Id}_{TX\otimes{\mathcal O}_X(-S)}$.
Thus the holomorphic Lie algebroid $(V,\,\phi)$ is logarithmically split (see Definition \ref{def-1}).
But $(V,\,\phi)$ is given to be logarithmically non-split. In view of this contradiction, we
conclude that the homomorphism $\zeta$ in \eqref{k8} vanishes identically.
This completes the proof.
\end{proof}

\begin{corollary}\label{cor1}
Let $(V,\, \phi)$ be a holomorphic Lie algebroid on $X$ such that $\phi_x\, =\, 0$ for all
$x\, \in\, S$ (as in \eqref{e8}), and $(V,\,\phi)$ is logarithmically non-split (see Definition
\ref{def-1}). Then any parabolic vector bundle $E_*$ on $X$ with parabolic divisor $S$ admits a
parabolic Lie algebroid connection for $(V,\, \phi)$.
\end{corollary}

\begin{proof}
Write $E_*$ as a direct sum of indecomposable parabolic vector bundles
$$
E_*\ =\ \bigoplus_{j=1}^d V^j_*.
$$
By Theorem \ref{thm1}, each indecomposable parabolic vector bundle $V^j_*$ admits a parabolic Lie
algebroid connection $D^j$ for $(V,\, \phi)$. Now $\bigoplus_{j=1}^d D_j$ is evidently
a parabolic Lie algebroid connection on $E_*$ for $(V,\, \phi)$.
\end{proof}

\section{Criterion for parabolic Lie algebroid connections}

\begin{lemma}
\label{lemma:composition}
Let $(V,\, \phi)$ and $(V',\, \phi')$ be holomorphic Lie algebroids on $X$
such that $\phi_x\,=\,0\,= \,\phi'_x$ for all $x\,\in\, S$. Let $\rho\, :\, V\, \longrightarrow\, V'$ be
holomorphic a Lie algebroid map. Let $E_*$ be a parabolic vector bundle and $D'\, : \, E\,\longrightarrow\,
E\otimes (V')^*$ a parabolic $(V',\, \phi')$--Lie algebroid connection on $E_*$. Then
$$D\ :=\ ({\rm Id}_E\otimes \rho^*) \circ D'\ : \ E\ \longrightarrow\ E\otimes V^*$$
is a parabolic $(V,\, \phi)$--Lie algebroid connection on $E*$.
\end{lemma}

\begin{proof}
Let $v$ (respectively, $s$) be a locally defined holomorphic section of $V$ (respectively $E$), and
let $f$ be a locally defined holomorphic function on $X$. Then, since $D'$ is a $(V', \, \phi')$--Lie algebroid
connection,
$$D_v(fs) \, = \, D'_{\rho(v)}(fs) \, = \, f D'_{\rho(v)}(s)+ \phi'(\rho(v))(f) s \, = \, f D_v(s) + \phi(v)(f) s,$$
so $D$ is a $(V,\, \phi)$--Lie algebroid connection. Let us verify that $D$ is a parabolic Lie algebroid connection.

Following \eqref{e2}, let $\mathcal{Q}\,=\,V^*/\phi^*(K_X)$ and $\mathcal{Q}'
\,=\,(V')^*/(\phi')^*(K_X)$. For each $x\in S$, let ${\mathcal S}_x \, : \, E_x\longrightarrow E_x\otimes
\mathcal{Q}_x$ and ${\mathcal S}_x' \, : \, E_x\longrightarrow E_x\otimes \mathcal{Q}_x'$ be the linear
maps defined by \eqref{e4} for $D$ and $D'$ respectively. Observe that $\phi^*\,=\,\rho^* \circ (\phi')^*$.
Thus, $\rho_x^*\, :\, (V')_x^* \longrightarrow V_x^*$ maps $(\phi')_x^*(K_X)_x$ to
$\rho_x^*\left((\phi')_x^*(K_X)_x\right)\, =\, \phi_x^*(K_X)_x$ and, therefore, it induces a map
$$\widehat{\rho}_x \, : \, (V')_x^*/(\phi')_x^*(K_X)_x\,=\,\mathcal{Q}_x'
\,\longrightarrow\, V_x^*/\phi_x^*(K_X)_x\,=\,\mathcal{Q}_x.$$ 
{}From the construction of the maps ${\mathcal S}_x$, ${\mathcal S}_x'$ and $\widehat{\rho}_x$ it follows
immediately that
\begin{equation}
\label{eq:composition1}
{\mathcal S}_x \ = \ ({\rm Id}_{E_x} \otimes \widehat{\rho}_x)\circ {\mathcal S}_x'.
\end{equation}
By Proposition \ref{prop1}, since $D'$ is parabolic, it follows that ${\mathcal S}_x'$ preserves the parabolic
filtration of $E_x$. Then, by \eqref{eq:composition1}, the homomorphism ${\mathcal S}_x$ also preserves the
parabolic structure, so $D$ is quasi-parabolic by Proposition \ref{prop1}. Let $$\widehat{D}_{x,i}\,\,:\,\,
E^i_x/E^{i+1}_x \,\, \longrightarrow\,\, (E^i_x/E^{i+1}_x)\otimes \mathcal{Q}_x$$ and
$\widehat{D}_{x,i}'\,\,:\,\, E^i_x/E^{i+1}_x \,\, \longrightarrow\,\, (E^i_x/E^{i+1}_x)\otimes \mathcal{Q}_x'$
be the maps described by \eqref{z2} for $D$ and $D'$ respectively. Then, by \eqref{eq:composition1},
\begin{equation}
\label{eq:composition2}
\widehat{D}_{x,i}\ =\ ({\rm Id_{E_x^i/E_x^{i+1}}}\otimes\widehat{\rho}_x) \circ \widehat{D}_{x,i}'.
\end{equation}
Let $\psi_x\, : \, \mathbb{C} \,\longrightarrow \,\mathcal{Q}_x$ and $\psi'_x\, : \, \mathbb{C} \,
\longrightarrow\, \mathcal{Q}_x'$ be the maps given by \eqref{e13} for $(V,\, \phi)$ and $(V',\, \phi')$
respectively. A direct computation shows that
\begin{equation}
\label{eq:composition3}
\psi_x \ = \ \widehat{\rho}_x \circ \psi_x'.
\end{equation}
Recall that $D'$ is a parabolic $(V',\, \phi')$--Lie algebroid connection, so $\widehat{D}_{x,i}'(v)\,=\,
v\otimes (\psi')_x^*(\alpha_i^x)$ for all $v\,\in\, E^i_x/E^{i+1}_x$. Thus, using \eqref{eq:composition2} and
\eqref{eq:composition3} we have
$$\widehat{D}_{x,i}(v) \, = \, v\otimes \widehat{\rho}_x((\psi')_x^*(\alpha_i^x)) \, = \, v\otimes \psi_x^*(\alpha_i^x)$$
for each $x\,\in\, S$. Consequently, $D$ is a parabolic $(V,\, \phi)$--Lie algebroid connection.
\end{proof}

Recall that the parabolic degree of a parabolic vector bundle $$E_*\ =\
(E,\, \{\{E^i_x\}_{i=1}^{\ell_x}\}_{x\in S},\, \{\{\alpha^x_i\}_{i=1}^{\ell_x}\}_{x\in S})$$ is defined to be
$${\rm pardeg}(E_*)\ :=\ \deg(E) + \sum_{x\in S} \sum_{i=1}^{l_x} \alpha_i^x (\dim E^i_x - \dim E^{i+1}_x).$$
\begin{theorem}\label{thm2}
Let $(V,\, \phi)$ be a holomorphic Lie algebroid on $X$ such that $\phi_x\, =\, 0$ for all
$x\, \in\, S$ (as in \eqref{e8}), and $(V,\,\phi)$ is logarithmically split (see Definition
\ref{def-1}). Then a parabolic vector bundle $E_*$ on $X$ admits a parabolic Lie algebroid connection
for $(V,\, \phi)$ if and only if the parabolic degree of every indecomposable component of $E_*$ is zero.
\end{theorem}

\begin{proof}
First assume that the parabolic degree of every indecomposable component of $E_*$ is zero.
Then $E_*$ admits a parabolic connection
$$
D_0\ :\ E\, \longrightarrow\ E\otimes K_X\otimes {\mathcal O}_X(S)
$$
(see \cite[p.~594, Theorem 1.1]{BL}), in other words, $D_0$ is a parabolic Lie algebroid connection on
$E_*$ for the holomorphic Lie algebroid $(TX\otimes {\mathcal O}_X(-S), \,\iota)$ (see \eqref{io}). Now
$$
({\rm Id}_E\otimes \phi^*_0)\circ D_0\ :\ E\, \longrightarrow\ E\otimes V^*,
$$
where $\phi^*_0$ is the dual of the homomorphism $\phi_0$ in \eqref{e9}, is a
parabolic Lie algebroid connection on $E_*$ for $(V,\, \phi)$ by Lemma \ref{lemma:composition}.

To prove the converse, assume that $E_*$ admits a parabolic Lie algebroid connection
$$
D\ :\ E\ \longrightarrow\ E\otimes V^*
$$
for $(V,\, \phi)$. Since $(V,\,\phi)$ is a logarithmically split holomorphic Lie algebroid, there is
a homomorphism
$$\sigma\ :\ TX\otimes{\mathcal O}_X(-S)\ \longrightarrow\ V$$
such that $\phi_0\circ\sigma\,=\, {\rm Id}_{TX\otimes{\mathcal O}_X(-S)}$, 
where $\phi_0$ is the homomorphism in \eqref{e9} (see \eqref{s}). Now
$$
({\rm Id}_E\otimes \sigma^*)\circ D \ :\ E\, \longrightarrow\ E\otimes K_X\otimes {\mathcal O}_X(S)
$$
is a parabolic connection on $E_*$, in other words, $({\rm Id}_E\otimes \sigma^*)\circ D$
is a Lie algebroid connection on
$E_*$ for the holomorphic Lie algebroid $(TX\otimes {\mathcal O}_X(-S), \,\iota)$.
This implies that the parabolic degree of every indecomposable component of $E_*$ is zero
\cite[p.~594, Theorem 1.1]{BL}.
\end{proof}

Combining Corollary \ref{cor1} and Theorem \ref{thm2} we have the following:

\begin{corollary}\label{cor2}
Let $(V,\, \phi)$ be a holomorphic Lie algebroid on $X$ satisfying the condition 
$\phi_x\, =\, 0$ for all $x\, \in\, S$.
A parabolic vector bundle $E_*$ on $X$ admits a parabolic Lie algebroid connection
for $(V,\, \phi)$ if and only if at least one of the following two conditions holds:
\begin{enumerate}
\item The holomorphic Lie algebroid $(V,\, \phi)$ is logarithmically non-split.

\item The parabolic degree of every indecomposable component of $E_*$ is zero.
\end{enumerate}
\end{corollary}

We will finish by including some comments on the existence of flat parabolic and quasi-parabolic
$(V,\, \phi)$--Lie algebroid connections. Given a holomorphic Lie algebroid connection $D \, : E\, \longrightarrow\, E\otimes V^*$,
and a locally defined holomorphic section $v$ of $V$, define $D_v$ as the locally defined $\mathbb{C}$-linear
endomorphism of $E$ induced by contracting the image of $D$ with $v$. Recall (see \cite{To2,ABKS}) that a
Lie algebroid is called flat or integrable if
$$D_{[v,w]} \ = \ [D_v,\ D_w]$$
for each pair of locally defined holomorphic sections $v,\,w$ of $V$. For any holomorphic
Lie algebroid connection $D$ on $E$, define the curvature of $D$ as the homomorphism
$F_D\,\in\, H^0\left(X,\, {\rm End}(E) \otimes \wedge^2 V^*\right )\, =\,
H^0\left ({\rm Hom} ( \wedge^2 V , \, {\rm End}(E))\right)$ which maps each local section $v\wedge w$ of
$\wedge^2V$ to
$$F_D(v\wedge w) \ =\ [D_v,\ D_w] - D_{[v,w]}.$$
In particular, $D$ is flat if and only if $F_D\, =\, 0$.

\begin{remark}\label{rmk:integrableRk1}
If $V$ is a line bundle, then $F_D\,=\,0$ for every $D$, since $\wedge^2V\, = \, 0$. Thus, if $V$ is a
line bundle, then every $(V,\, \phi)$--Lie algebroid connection is flat.
\end{remark}

\begin{lemma}
\label{lemma:composition2}
Let $(V,\, \phi)$ and $(V',\, \phi')$ be holomorphic Lie algebroids on $X$ such that $\phi_x\,=\,0$ and
$\phi'_x\,=\,0$ for all $x\,\in\, S$, and let $\rho\, :\, V\,\longrightarrow\, V'$ be a
holomorphic Lie algebroid map. Let $E_*$ be a parabolic vector bundle and $D'\, : \, E\,\longrightarrow\,
E\otimes (V')^*$ a flat parabolic $(V',\, \phi')$--Lie algebroid connection on $E_*$. Then
$D\, =\, ({\rm Id}_E\otimes \rho^*) \circ D'\, : \, E\, \longrightarrow\, E\otimes (V)^*$ is a flat parabolic
$(V,\, \phi)$--Lie algebroid connection on $E_*$.
\end{lemma}

\begin{proof}
By Lemma \ref{lemma:composition}, $D$ is a parabolic $(V,\, \phi)$--Lie algebroid connection. Let $v$ and $w$ be
locally defined holomorphic sections of $V$. Since $\rho$ preserves the Lie bracket operations
of $V$ and $V'$, we have
$$[D_v,\, D_w]\ =\ [D'_{\rho(v)}, \, D'_{\rho(w)}] \ =\ D'_{[\rho(v), \rho(w)]} \ = \ D'_{\rho([v,w])} \
 = \ D_{[v,w]}.$$
Thus, $D$ is flat.
\end{proof}

\begin{corollary} \label{cor:integrable}
Let $(V,\, \phi)$ be a holomorphic Lie algebroid such that $\phi_x\,=\,0$ for all $x\,\in\, S$. Let $E_*$
be a parabolic vector bundle. Then the following three statements hold.
\begin{enumerate}
\item $E_*$ always admits a flat quasi-parabolic $(V,\, \phi)$--Lie algebroid connection.

\item If ${\rm image}(\phi)\, \neq \, TX\otimes \mathcal{O}_X(-S)$ or ${\rm image}(\phi)\, = \,
TX\otimes \mathcal{O}_X(-S)$, and the parabolic degree of every indecomposable component of $E_*$ is zero, then
$E_*$ admits a flat parabolic $(V,\, \phi)$--Lie algebroid connection.

\item If ${\rm image}(\phi)\, = \, TX\otimes \mathcal{O}_X(-S)$ and $(V,\, \phi)$ is logarithmically split, then
$E_*$ admits a flat parabolic $(V,\, \phi)$--Lie algebroid connection if and only if the parabolic degree of
every indecomposable component of $E_*$ is zero.
\end{enumerate}
\end{corollary}

\begin{proof}
If $\phi\, =\, 0$, then $D\,=\,0$ is a flat parabolic $(V,\, 0)$-Lie algebroid connection on $E_*$. Suppose
that $\phi\, \ne \, 0$. Set $L\,=\,{\rm image}(\phi)\, \subseteq \, TX\otimes \mathcal{O}_X(-S)$. Since
$\phi$ is a Lie algebroid map (see, for instance, \cite[Remark 2.1]{ABKS}), it follows that $L$ is a line bundle
inheriting a holomorphic Lie algebroid structure from $TX$, with the inclusion map $\iota\, :\, L\,
\hookrightarrow\, TX$ being the anchor map. Observe that since $\phi_x\,=\,0$ for all $x\,\in \,S$,
we have $\iota_x\,=\,0$ for all $x\,\in\, S$. Also, let us call $\phi_L\, : \, V\longrightarrow L$ the Lie
algebroid map induced by $\phi$.

If $L \, \ne \, TX\otimes \mathcal{O}_X(-S)$, then $(L,\, \iota)$ is not logarithmically split. Thus, by
Corollary \ref{cor1}, $E_*$ admits a parabolic $(L,\, \iota)$--Lie algebroid connection
$D'\, : \, E \,\longrightarrow\, E\otimes L^*$. Since $L$ is a line bundle, by Remark \ref{rmk:integrableRk1},
this $D'$ is flat. Define
\begin{equation}
\label{eq:integrable1}
D \ :=\ ({\rm Id}_E \otimes \phi_L^*) \circ D' \ : \ E\ \longrightarrow\ E \otimes V^*.
\end{equation}
By Lemma \ref{lemma:composition2}, $D$ is a flat parabolic $(V,\, \phi)$--Lie algebroid connection.

Suppose now that $L\,=\,TX\otimes \mathcal{O}_X(-S)$. If the parabolic degree of every indecomposable component 
of $E_*$ is zero, then $E_*$ admits a parabolic $(L,\, \iota)$--Lie algebroid 
connection $D'$ (see \cite[p.~594, Theorem 1.1]{BL}), and repeating the previous argument, the map $D$
from \eqref{eq:integrable1} yields --- once again --- a flat 
parabolic $(V,\, \phi)$--Lie algebroid connection on $E_*$. This proves statement (2).
Statement (3) is a consequence of (2) and Theorem \ref{thm2}.

Let us now prove statement (1). We can assume without loss of generality that $L\,=\,
TX\otimes \mathcal{O}_X(-S)$ and that $E_*$ is indecomposable, as it is enough to build a quasi-parabolic
connection on each indecomposable component of $E_*$ in order to build one on $E_*$. Choose real numbers
$\beta_i^x$ for each $x\in S$ and $1\,\le\, i\,\le\, l_x$ such that
\begin{equation}
\label{eq:integrable2}
\deg(E) + \sum_{x\in S} \sum_{i=1}^{l_x} \beta_i^x (\dim E^i_x - \dim E^{i+1}_x) \ = \ 0.
\end{equation}
Let $E'_*$ denote the quasi-parabolic vector bundle underlying $E_*$ endowed with the
new system of weights $\beta$. Even
though parabolic systems of weights are usually taken to be lying between $0$ and $1$, the arguments in \cite{BL}
remain valid for any arbitrary system of weights, as long as \eqref{eq:integrable2} is satisfied. In
particular, from \cite[p.~594, Theorem 1.1]{BL}, the parabolic vector bundle $E_*'$ admits a quasi-parabolic
$(TX\otimes \mathcal{O}_X(-S),\, \iota)$--Lie algebroid connection $D'$, whose residues at the parabolic
points act diagonally as multiplication by the weights $\beta$. Then, Lemma \ref{lemma:composition2}, implies
that the map $D$ from \eqref{eq:integrable1} is a flat quasi-parabolic connection on $E_*$.
\end{proof}

\section*{Acknowledgements}

D.A. was partially supported by MCIN/AEI/10.13039/501100011033 grant PID2022-142024NB-I00. I.B. is partially 
supported by J. C. Bose Fellowship (JBR/2023/000003). A.S. is partially supported by ANRF/ARGM/2025/000670/MTR.


\begin{thebibliography}{ZZZZZZ}

\bibitem[Al]{Al} D. Alfaya, Moduli space of parabolic $\Lambda$-modules over a curve, 
arXiv:1710.02080.

\bibitem[AO]{AO} D. Alfaya and A. Oliveira, Lie algebroid connections, twisted Higgs bundles
and motives of moduli spaces, {\it J. Geom. Phys.} {\bf 201} (2024), 105195.

\bibitem[ABKS]{ABKS} D. Alfaya, I. Biswas, P. Kumar and A. Singh, Parabolic vector bundles and Lie
algebroid connections, \textit{Canadian Jour. Math.}, https://doi.org/10.4153/S0008414X25101983.

\bibitem[At]{At} M. F. Atiyah, Complex analytic connections in fibre bundles, {\it
Trans. Amer. Math. Soc.} {\bf 85} (1957), 181--207.

\bibitem[BL]{BL} I. Biswas and M. Logares, Connections on parabolic vector bundles over curves, {\it 
Internat. J. Math.} {\bf 22} {(4)} (2011) 593--602.

\bibitem[CM]{CM} J. Cort{\'{e}}s and E. Mart{\'{\i}}nez,
Mechanical control systems on Lie algebroids,
{\em {IMA} Jour. Math. Control Information} {\bf 21} (2004), 457--492.

\bibitem[De]{De} P.~Deligne, {\em \'{E}quations diff\'{e}rentielles \`a points singuliers
r\'{e}guliers}, Lecture Notes in Mathematics, Vol. 163. Springer-Verlag, Berlin-New York, 1970.

\bibitem[ELW]{ELW} S. Evens, J.-H. Lu and A. Weinstein, Transverse measures, the modular
class and a cohomology pairing for Lie algebroids, {\it Quart. Jour. Math.} {\bf 50}
(1999), 417--436.

\bibitem[GH]{GH} P. Griffiths and J. Harris, {\it Principles of algebraic geometry},
Pure and Applied Mathematics, Wiley-Interscience, New York, 1978.

\bibitem[LM]{LM} S. Lazzarini and T. Masson, Connections on Lie algebroids and on
derivation-based noncommutative geometry, {\em Jour. Geom. Phy.} {\bf 62} (2012), 387--402.

\bibitem[MY]{MY} M. Maruyama and K. Yokogawa, Moduli of parabolic stable
sheaves, \textit{Math. Ann.} \textbf{293} (1992) 77--99.

\bibitem[MS]{MS} V. B. Mehta and C. S. Seshadri, Moduli of vector bundles on curves with
parabolic structures, {\it Math. Ann.} \textbf{248} (1980), 205--239.

\bibitem[PW]{PW} \'{A}. del Pino and A. Witte, Regularisation of Lie algebroids
and applications, {\it Jour. Geom. Phys.} {\bf 194} (2023), 105023.

\bibitem[To1]{To1} P. Tortella, {\it {$\Lambda$}-modules and holomorphic Lie algebroids}, PhD thesis,
Scuola Internazionale Superiore di Studi Avanzati (2011).

\bibitem[To2]{To2} P. Tortella, $\Lambda$-modules and holomorphic Lie algebroid connections,
{\it Cent. Eur. J. Math.} {\bf 10} (2012), 1422--1441.

\bibitem[Yok]{Yok} K. Yokogawa, Infinitesimal deformation of parabolic Higgs sheaves,
\textit{Internat. J. Math.} \textbf{6} (1995) 125--148.

\end{thebibliography}
\end{document}